\documentclass[english,12pt,oneside]{amsproc}
\usepackage{a4wide}

\usepackage[T2C]{fontenc}
\usepackage[cp866]{inputenc}
\usepackage[english]{babel}
\usepackage{amsfonts,amssymb,mathrsfs,amscd}
\usepackage{mathdots}
\usepackage{latexsym}
\usepackage[matrix,arrow,curve]{xy}
\usepackage{mathabx}
\usepackage{color}
\usepackage{pbox}
\usepackage{tikz}
\usepackage{hyperref}
\newcounter{stmcounter}[section]
\numberwithin{equation}{section}

\newcounter{thmMaincounter}

\theoremstyle{plain}
\newtheorem{cor}[stmcounter]{Corollary}

\newtheorem{thm}[stmcounter]{Theorem}
\newtheorem{thmM}[thmMaincounter]{Theorem}

\newtheorem{prop}[stmcounter]{Proposition}
\newtheorem{lem}[stmcounter]{Lemma}

\theoremstyle{definition}
\newtheorem{defin}[stmcounter]{Definition}
\newtheorem{con}[stmcounter]{Construction}
\newtheorem{ex}[stmcounter]{Example}
\newtheorem{rem}[stmcounter]{Remark}

\theoremstyle{remark}

\DeclareMathOperator{\pt}{pt}

\DeclareMathOperator{\codim}{codim}

\DeclareMathOperator{\Hom}{Hom}
\DeclareMathOperator{\rk}{rk}
\DeclareMathOperator{\odd}{odd}

\DeclareMathOperator{\com}{compl}

\DeclareMathOperator{\Tor}{Tor}
\DeclareMathOperator{\ord}{ord}
\DeclareMathOperator{\Top}{Top}
\DeclareMathOperator{\Ab}{Ab}
\DeclareMathOperator{\hocolim}{hocolim}
\DeclareMathOperator{\colim}{colim}
\DeclareMathOperator{\cat}{cat}
\DeclareMathOperator{\Funct}{Funct}
\DeclareMathOperator{\str}{star}

\DeclareMathOperator{\Fl}{Fl}
\DeclareMathOperator{\Gr}{Gr}
\DeclareMathOperator{\Sgr}{Sgr}

\def\G{\Gamma}
\def\GX{\Gamma(X)}
\def\GY{\Gamma(Y)}
\def\ZQ{R}

\def\bv{\bar{\varphi}}
\def\Co{\mathbb C}
\def\Ro{\mathbb R}
\def\Qo{\mathbb Q}
\def\Zo{\mathbb Z}

\newcommand{\HP}{\mathbb{H}P}
\newcommand{\CP}{\mathbb{C}P}

\newcommand{\ca}[1]{\mathcal{#1}}
\newcommand{\Hr}{\widetilde{H}}
\newcommand{\dd}{\partial}

\begin{document}

\title{How is a graph not like a manifold?}

\author{A.\,A.~Ayzenberg}
\address{Faculty of computer science, National Research University Higher School of Economics, Russian Federation}
\email{ayzenberga@gmail.com}

\author{M.~Masuda}
\address{Osaka City University Advanced Mathematical Institute, Japan, and Faculty of Computer Science, National Research University Higher School of Economics, Russian Federation}
\email{mikiyamsd@gmail.com}

\author{G.\,D.~Solomadin}
\address{Faculty of computer science, National Research University Higher School of Economics, Russian Federation}
\email{grigory.solomadin@gmail.com}




\thanks{The article was prepared within the framework of the HSE University Basic Research Program}


\begin{abstract}
For an equivariantly formal action of a compact torus $T$ on a smooth manifold $X$ with isolated fixed points we investigate the global homological properties of the graded poset $S(X)$ of face submanifolds. We prove that the condition of $j$-independency of tangent weights at each fixed point implies $(j+1)$-acyclicity of the skeleta $S(X)_r$ for $r>j+1$. This result provides a necessary topological condition for a GKM graph to be a GKM graph of some GKM manifold. We use particular acyclicity arguments to describe the equivariant cohomology algebra of an equivariantly formal manifold of dimension $2n$ with an $(n-1)$-independent action of $(n-1)$-dimensional torus, under certain colorability assumptions on its GKM graph. This description relates the equivariant cohomology algebra to the face algebra of a simplicial poset. Such observation underlines certain similarity between actions of complexity one and torus manifolds.
%
\end{abstract}

\maketitle

\section{Introduction}\label{secIntro}

Toric topology studies actions of a compact torus $T^k$ on closed smooth manifolds $X^{2n}$ in terms of the related combinatorial structures. The classical examples are given by smooth toric varieties which are classified by their simplicial fans, and their topological analogues --- quasitoric manifolds classified by characteristic pairs. In both cases the torus $T^n$ acts on a manifold $X^{2n}$ with $H^{\odd}(X^{2n})=0$, and it happens that the poset $S(X)$ of $T^n$-invariant submanifolds in $X^{2n}$ is a certain cell subdivision of a topological disc. The poset $S(X)$ has nice acyclicity properties: not only $S(X)$ is acyclic (which is obvious since it has the greatest element), but also its skeleta, the links of its simplices, and other natural related objects are acyclic as well. These acyclicity properties are intimately related to the Cohen--Macaulay property of Stanley--Reisner algebras, which are isomorphic to equivariant cohomology of~$X$.

In this paper we study torus actions of arbitrary complexity. For a general action of $T=T^k$ on $X=X^{2n}$ having isolated fixed points, we consider the poset of face submanifolds $S(X)$. In~\cite{AyzCherep}, the local structure of $S(X)$ was studied: it was proved that $S(X)_{\geqslant s}$ is a geometric lattice for any element $s\in S(X)$, in particular, such ``local'' poset is acyclic due to the result of Bj\"{o}rner~\cite{Bjorner}.

In this paper we concentrate on global topological structure of $S(X)$, however, restrict the class of actions under consideration. We assume that torus actions satisfy the following properties.

\begin{enumerate}
  \item The fixed point set $X^T$ is finite and nonempty.
  \item The action is \emph{equivariantly formal}. With item 1 satisfied, equivariant formality is equivalent to the condition $H^{\odd}(X)=0$.
  \item We fix an integer $j$, and call the action \emph{$j$-independent}, if every $\leqslant j$ tangent weights, at every fixed point $x\in X^T$ are linearly independent. In the previous works~\cite{AyzMasEquiv,AyzCher} we used a different terminology: actions with this property were called actions in $j$-general position.
\end{enumerate}

With these properties satisfied, we prove the following result.

\begin{thmM}\label{thmMainIntroAcycl}
Assume an action of $T$ on $X$ is equivariantly formal and $j$-inde\-pen\-dent. Then the following acyclicity conditions hold.
\begin{enumerate}
  \item For any integer $r>0$, the $r$-skeleton $S(X)_r=\{t\in S(X)\mid \rk t\leqslant r\}$ is $\min(\dim S(X)_r-1,j+1)$-acyclic.
  \item For any element $s\in S(X)$, the lower order ideal $S(X)_{<s}=\{t\in S(X)\mid t<s\}$ is $\min(\dim S(X)_{<s}-1,j+1)$-acyclic.
\end{enumerate}
\end{thmM}

This theorem gives a necessary condition for a GKM graph to be a GKM graph of some (equivariantly formal) GKM manifold for $j$ large enough. Indeed, if $\G$ is a $j$-independent GKM graph, then every $j-1$ edges adjacent to a vertex span a face, as follows e.g. from the results of~\cite{AyzCherep}. Therefore the poset $S(X)_{j-1}$ can be reconstructed only from a GKM graph data. If the poset is not $(j-2)$-acyclic, then the graph does not correspond to any equivariantly formal manifold.

It should be noted that the problem of reconstruction a manifold with the given fixed points data, in particular with the given equivariant 1-skeleton, is well known in equivariant topology, see~\cite{CGK}. The classical approach to solve this problem is via Atiyah--Bott--Berline--Vergne formula (the instance of the localization formula in cohomology). Sometimes it is possible to reconstruct a manifold with torus action geometrically --- by extending the action from lower dimensional strata to higher dimensions. This approach, however, does not guarantee that the resulting manifold is equivariantly formal, and that GKM theory itself is applicable. This is an important issue we wanted to address in our paper.

As an application of Theorem~\ref{thmMainIntroAcycl}, we describe equivariant cohomology algebra for complexity one actions in general position. Acyclicity arguments are applied to extend the study of complexity one actions in general position, started in~\cite{AyzMasEquiv}. An action of $T^{n-1}$ on $X^{2n}$ is called \emph{an action in general position}, if it is $(n-1)$-independent.

\begin{thmM}\label{thmMainIntroFaceRing}
Consider an equivariantly formal action of $T^{n-1}$ on $X^{2n}$ in general position. Assume that $n\geqslant 5$, $\pi_1(X)=1$, and the GKM graph of the action is bipartite. Then there exists a simplicial poset $S(\GX)$ and a regular linear element $\eta$ in the face ring $\Qo[S(\GX)]$, such that
\[
H^*_T(X;\Qo)\cong \Qo[S(\GX)]/(\eta)
\]
as $H^*(BT;\Qo)$-algebras.
\end{thmM}

The poset $S(\GX)$ is constructed from $S(X)$ by adding some extra elements and reversing the order. The poset $S(\GX)$ is Gorenstein*: this follows from the Gorenstein property of the algebra $H^*_T(X;\Qo)$, as explained in detail in Remark~\ref{remGorenstein}. In some sense, Theorem~\ref{thmMainIntroFaceRing} tells that under certain assumptions complexity one actions in general position behave much like restrictions of $T^n$-actions on $X^{2n}$ to actions of generic subtori $T^{n-1}\subset T^n$, at least from the viewpoint of equivariant cohomology.

The paper has the following structure. In Section~\ref{secFaces} we recall the necessary definitions: equivariant formality, face submanifolds, and faces of a torus action. In Section~\ref{secAcyclicity} we prove several statements about homological properties of orbit spaces of torus actions. Most of the arguments there follow the lines of~\cite{AyzMasEquiv} however we recall the key arguments. In Section~\ref{secPosets} we prove Theorem~\ref{thmMainIntroAcycl}. In Section~\ref{secGKMcomb} we provide all necessary definitions from the GKM theory and prove a combinatorial statement, that for $n\geqslant 5$, the $n$-valent GKM graph of a complexity one action in general position is bipartite if and only if it is $n$-colorable. This is a generalization of the result of Joswig~\cite{Joswig} which asserts a similar proposition for the edge skeleta of simple polytopes (hence can be viewed as a statement about GKM graphs of complexity zero). The ability to properly color a GKM graph implies that the graph determines a simplicial poset in a way similar to complexity zero. In Section~\ref{secFaceRings} we recall the notion of the face ring of a simplicial poset, and derive Theorem~\ref{thmMainIntroFaceRing} from the GKM description of $H^*_T(X)$.

\section{Preliminaries}\label{secFaces}

\subsection{Orbit type filtrations}

In this section we define the faces of an action and list their main properties.

Let a torus $T=T^k$ act on a topological space $X$ which is always assumed connected. Let $\Sgr(T)$ denote the set of all closed subgroups of $T$. For a point $x\in X$, $T_x\in \Sgr(T)$ denotes the stabilizer (the stationary subgroup) of $x$, and $Tx\subset X$ --- the orbit of $x$. In the following we assume that $X$ is a $T$-CW-complex (see~\cite[Def.1.1]{AdDav}). In particular, this holds for smooth torus actions on smooth manifolds. 


\begin{con}\label{conExactPartition}
For an action of $T$ on $X$ we define the \emph{fine subdivision} of $X$ by orbit types:
\[
X=\bigsqcup_{H\in \Sgr(T)}X^{(H)},
\]
where $H$ is a closed subgroup of $T$, and $X^{(H)}=\{x\in T\mid T_x=H\}$. Moreover, for $H\in \Sgr(T)$ define
\[
X^H=\bigsqcup_{\tilde{H}\supseteq H}X^{(\tilde{H})}=\left\{x\in X\mid hx=x\,\,\forall h\in H\right\}.
\]
Therefore $X^H$ is the set of $H$-fixed points of $X$.
\end{con}

Each closed subgroup of a compact torus $T\cong T^k$ is isomorphic to a direct product of some torus (the continuous component) and finite abelian group (the finite component). We say that a $T$-action on $X$ has connected stabilizers, if all stabilizers $T_x$ are connected, i.e. finite components are trivial.

When working with homology we will follow the assumption that for actions with connected stabilizers the coefficients are taken in $\Zo$ (or any field), but in general the coefficients are taken in $\Qo$.

\begin{con}\label{conFiltrationOrbitType}
For a $T$-action on a topological space $X$ consider the fil\-tra\-ti\-on
\begin{equation}\label{eqEquivFiltrGeneralX}
X_0\subset X_1\subset\cdots \subset X_k
\end{equation}
where $X_i$ is the union of all orbits of the action having dimension $\leqslant i$. In other words,
\[
X_i=\{x\in X\mid \dim T_x\geqslant k-i\}=\bigsqcup_{H\in \Sgr(T),\dim H\geqslant k-i}X^{(H)}
\]
according to the natural homeomorphism $Tx\cong T^k/T_x$. Filtration~\eqref{eqEquivFiltrGeneralX} is called \emph{the orbit type filtration}, and $X_i$ \emph{the equivariant $i$-skeleton} of $X$. Each $X_i$ is $T$-stable. Orbit type filtration induces the filtration on the orbit space $Q=X/T$:
\begin{equation}\label{eqEquivFiltrGeneralQ}
Q_0\subset Q_1\subset\cdots \subset Q_k,\quad\mbox{where }Q_i=X_i/T
\end{equation}
\end{con}

\begin{rem}
If $y\in Q$ is an orbit of the action, then stabilizer subgroup $T_y$ is defined as the stabilizer subgroup $T_x$ for any representative $x\in y$. This is well-defined due to commutativity of the torus. In the following, when we speak about fixed points of the action, we abuse the notation by denoting with the same letter a fixed point and its image in the orbit space.
\end{rem}

\subsection{Smooth actions}\label{subsecSmoothActions}

\begin{defin}
The lattice $N=\Hom(T^k,S^1)\cong \Zo^k$ is called \emph{the weight lattice}, and its dual lattice $N^*=\Hom(S^1,T^k)$ is called \emph{the lattice of 1-dimensional subgroups}.
\end{defin}

There are canonical isomorphisms
\begin{equation}\label{eqWeightSubgrpLatticeHomologyIdentific}
\begin{split}
\Hom(T^k,S^1)&\cong H^1(T^k;\Zo)\cong H^2(BT^k;\Zo),\\ \Hom(S^1,T^k)&\cong H_1(T^k;\Zo)\cong H_2(BT^k;\Zo).
\end{split}
\end{equation}

Let $X$ be a smooth closed connected orientable manifold, and $T$ acts smoothly and effectively on $X$. If $x\in X^T$ is a fixed point, we have an induced representation of $T$ in the tangent space $\tau_xX$ called \emph{the tangent representation}. Let $\alpha_{x,1},\ldots,\alpha_{x,n}\in \Hom(T^k,S^1)\cong \Zo^{k}$ be the weights of the tangent representation at $x$, which means by definition that
\[
\tau_xX\cong V(\alpha_{x,1})\oplus\cdots\oplus V(\alpha_{x,n})\oplus \Ro^{\dim X-2n},
\]
where $V(\alpha)$ is the standard 1-dimensional complex representation given by $tz=\alpha(t)\cdot z$, $z\in \Co$, and the action on $\Ro^{\dim X-2n}$ is trivial (see~\cite[Cor.I.2.1]{Hsiang}). It is assumed that all weight vectors $\alpha_{x,i}$ are nonzero since otherwise the corresponding summands contribute to $\Ro^{\dim X-2n}$. If there is no $T$-invariant complex structure on $X$, then there is an ambiguity in the choice of signs of vectors $\alpha_i$. For the statements of this paper the choice of signs is nonessential. We can also assume that the weight vectors $\alpha_{x,1},\ldots,\alpha_{x,n}$ linearly span the weight lattice $\Hom(T^k,S^1)$ (otherwise there would exist an element $\lambda$ of the dual lattice $\Hom(S^1,T^k)$ such that $\langle\alpha_{x,1},\lambda\rangle=0$, which implies that the corresponding 1-dimensional subgroup $\lambda$ lies in the noneffective kernel). This observation implies that, if the action has fixed points, we have
\begin{equation}\label{eqIneqForComplexity}
\dim X\geqslant 2n\geqslant 2k.
\end{equation}

Each fixed point $x\in X^T$ has a neighborhood equivariantly diffeomorphic to the tangent representation $\tau_xX$. In particular, $x\in X^T$ is isolated if and only if $\dim X=2n$ and all tangent weights $\alpha_{x,1},\ldots,\alpha_{x,n}$ are nonzero.

\begin{defin}\label{definComplexity}
Let $T$ act effectively on a smooth manifold $X$, and the fixed point set $X^T$ is finite and nonempty. The nonnegative integer $\com X=\frac{1}{2}\dim X-\dim T$ is called \emph{the complexity of the action}.
\end{defin}

If the action is noneffective, the symbol $\com X$ denotes the complexity of the corresponding effective action: the action of the quotient by the noneffective kernel.

\begin{con}\label{conInvarSubmfds}
For each closed subgroup $H\subset T$, the subset $X^H$ is a closed smooth submanifold in $X$. This submanifold is $T$-stable as follows from the commutativity of $T$. A connected component of $X^H$ is called an \emph{invariant} submanifold.
\end{con}

\begin{con}\label{conFacesOrbitSpaceGeneral}
For a smooth $T$-action on $X$, consider the canonical pro\-jec\-tion $p\colon X\to Q$ to the orbit space, and the filtration~\eqref{eqEquivFiltrGeneralQ} on the orbit space. The closure of a connected component of $Q_i\setminus Q_{i-1}$ is called \emph{a face} $F$ if it contains at least one fixed point. The number $i$ is called \emph{the rank} of a face $F$, it is equal to the dimension of a generic $T$-orbit in $F$.
\end{con}

\begin{rem}
In the case of locally standard action of $T=T^n$ on $X=X^{2n}$, the orbit space $Q=X/T$ is a nice manifold with corners, so it has a naturally defined notion of faces. The preimages of faces of $Q$ are all the invariant submanifolds of $X$. Faces of $Q$ (in the sense of torus actions) are only those faces of $Q$ (in the sense of manifold with corners) which have vertices.
\end{rem}

\begin{rem}
In general, the notion of a face of an orbit space is determined by the action, so the face subdivision resembles an additional structure on the orbit space $Q$. Knowledge of topology of $Q$ itself is not sufficient to define the faces. Even in complexity~0, $Q$ is just a topological manifold with boundary so one can only distinguish whether the point of $Q$ is a free orbit or not depending on whether it lies in the interior or on the boundary. There are examples in higher complexity, when even free and non-free orbits can't be distinguished in the orbit space~\cite{AyzCompl}. However, we slightly abuse the notation by using the term ``face of the orbit space'' as if it were defined intrinsically by the topology of $Q$.
\end{rem}
%
%
%

The following lemma is proved in~\cite{AyzCherep}.

\begin{lem}\label{lemFaceSubmanifold}
For a $T$-action on $X$, the full preimage $X_F=p^{-1}(F)$ of any face $F\subset Q$ is an invariant submanifold. In particular, it is a smooth closed submanifold of $X$.
\end{lem}

\begin{defin}\label{definFaceSubmfd}
Let $F$ be a face of $Q=X/T$. The submanifold $X_F=p^{-1}(F)\subset X$ is called a \emph{face submanifold} corresponding to $F$.
\end{defin}

Notice that the definition of a face implies that each face submanifold necessarily has a $T$-fixed point. Some details and formalism about the notions of faces and face submanifolds can be found in the recent preprint~\cite{AyzCherep}.

\begin{con}\label{conPosetOfFaces}
It is known~\cite[Thm.5.11]{Dieck} that a smooth action of $T^k$ on a compact smooth manifold has only finite number of possible stabilizers. This implies that there exist only finite number of faces. The set of faces of $Q$ (equiv. face submanifolds of $X$) is partially ordered by inclusion and graded by ranks. We denote this poset by $S(X)$.
\end{con}

\begin{con}\label{conCommonStabilizer}
Let $F$ be a face and $X_F$ be the corresponding face submanifold. The action of $T$ on $X_F$ has noneffective kernel
\[
T_F=\{t\in T\mid tx=x\,\,\forall x\in X_F\},
\]
which we call a common stabilizer of points from $F$. The number $\dim T/T_F$ equals the rank of $F$.

The effective action of $T/T_F$ on $X_F$ satisfies the general assumption from Definition~\ref{definComplexity}: its fixed point set is nonempty (due to the definition of a face) and finite (because it is a subset of $X^T$, and the latter is assumed finite). Therefore, the induced complexity $\com X_F$ is well defined:
\[
\com X_F = \dim X_F - \dim(T/T_F) = \dim X_F - \rk F.
\]
\end{con}

Notice that each face $F$ of rank $r$ is a subset of the filtration term $Q_r$. It is tempting to say that $Q_r$ is the union of all faces of rank $r$. However, this may be false in general, the example when $Q_r$ is not the union of $r$-dimensional faces can be found in~\cite[Fig.1]{AyzCherep}. Nevertheless, there is no such problem for equivariantly formal actions as explained in Remark~\ref{remEqFormalSkeleta} below.

\section{Acyclicity for independent actions}\label{secAcyclicity}

\subsection{Equivariant formality}
The definitions and statements of this subsection are well known and given for convenience of the reader.

For $T\cong T^k$ we have the universal principal $T$-bundle $ET\to BT$, $BT\simeq (\CP^\infty)^k$. The $T$-action on $X$ determines the Borel construction $X_T=X\times_T ET$, the Serre fibration $p\colon X_T\stackrel{X}{\to} BT$, and the equivariant cohomology ring $H^*_T(X;R)=H^*(X_T;R)$ which is a module over the polynomial ring $H^*(BT;R)\cong R[k]=R[v_1,\ldots,v_k]$ (the module structure is induced by $p^*$). The fibration $p$ induces the Serre spectral sequence:
\begin{equation}\label{eqSerreSS}
E_2^{p,q}\cong H^p(BT^k;R)\otimes H^q(X;R)\Rightarrow H^{p+q}_T(X;R).
\end{equation}

\begin{defin}\label{definEquivFormality}
The $T$-action on $X$ is called \emph{cohomologically equivariantly formal (over $R$)} in the sense of Goresky--Kottwitz--MacPherson, if the spectral sequence~\eqref{eqSerreSS} collapses at $E_2$.
\end{defin}

We simply call such actions and spaces \emph{equivariantly formal}. The definition of equivariant formality was given by Goresky--Kottwitz--MacPherson in~\cite{GKM} in case of coefficients in~$\Ro$.

%

For convenience, we provide a lemma proved in~\cite{FP} which gives equivalent reformulation of equivariant formality.

\begin{lem}[\cite{FP}]\label{lemEquivFormalCriterion}
The following conditions are equivalent (all coefficients are either in $\Zo$ or a field):
\begin{enumerate}
 \item A $T$-action on $X$ is equivariantly formal.
 \item The inclusion of fiber into the Borel construction $\iota\colon X\to X_T$ induces the surjective map
 $\iota^*\colon H^*_T(X)\to H^*(X)$.
 \item The homomorphism $H^*_T(X)\otimes_{H^*(BT)}R\to H^*(X)$, induced by $\iota^*$, is an isomorphism.
 \item $\Tor^j_{H^*(BT)}(H^*_T(X);R)=0$ for all $j>0$.
 \item $\Tor^1_{H^*(BT)}(H^*_T(X);R)=0$.
\end{enumerate}
If the coefficients are in a field, then these conditions are equivalent to freeness of the $H^*(BT)$-module $H^*_T(X)$.
\end{lem}

Over $\Zo$, the freeness of $H^*(BT)$-module is strictly stronger than equivariant formality (see~\cite{FPOsaka}). 

In the following we only consider the actions with isolated fixed points. In this case, the characterization of equivariant formality becomes easier.

\begin{lem}[{\cite[Lm.2.1]{MasPan}}]\label{lemEquivFormFixedPoints}
Consider a smooth $T$-action on $X$, such that $X^T$ is finite and nonempty. Then the following conditions are equivalent
\begin{enumerate}
 \item The $T$-action on $X$ is cohomologically equivariantly formal.
 \item $H^{\odd}(X)=0$.
 \item $H^*_T(X)$ is a free $H^*(BT)$-module.
\end{enumerate}
If either of these conditions hold, we have an isomorphism of graded $H^*(BT)$-modules $H^*_T(X)\cong H^*(BT)\otimes H^*(X)$.
\end{lem}

Another important statement asserts that equivariant formality is inherited by invariant submanifolds.

\begin{lem}[{\cite[Lem.2.2]{MasPan}}]\label{lemInvarIsFormal}
Let $T$ act on $X$, and $Y$ be an invariant submanifold (a connected component of $X^H$ for some $H\in \Sgr(T)$). Then condition $H^{\odd}(X)=0$ implies $H^{\odd}(Y)=0$ and $Y^T\neq\varnothing$.
\end{lem}

\begin{rem}\label{remEqFormalSkeleta}
Lemma~\ref{lemInvarIsFormal} implies that the requirement for a face submanifold (and for a face) to contain a fixed point is automatically satisfied for equivariantly formal actions with isolated fixed points. This, in turn, implies that the equivariant $r$-skeleton $Q_r$ is exactly the union of all faces $F$ of rank $r$, when we deal with equivariantly formal actions.
\end{rem}

\subsection{$j$-independent actions}

\begin{defin}\label{definJgeneral}
A $T$-action on a manifold $X$ is \emph{$j$-independent}\footnote{In~\cite{AyzMasEquiv} and~\cite{AyzCher} we also called such actions the actions in $j$-general position.} if, for any fixed point $x\in X^T$, any $\leqslant j$ of the tangent weights $\alpha_{x,1},\ldots,\alpha_{x,n}\in \Hom(T,S^1)\cong \Zo^k$ are linearly independent over $\Qo$.
\end{defin}
%

\begin{rem}\label{remCom0MaxGenerality}
If $\com X=0$, that is $T=T^n$ acts on $X=X^{2n}$, then, at each fixed point $x\in X^T$ we have $n$ weights $\alpha_{x,1},\ldots,\alpha_{x,n}\in \Hom(T,S^1)\cong \Zo^n$ which linearly span $\Qo^n$. Hence these actions are $n$-independent. It is also true that this action is $\infty$-independent since any subset of the set of tangent weights is linearly independent. We only work with finite-dimensional manifolds, so this notation is just a matter of formalism; however, some statements about $j$-independent actions make perfect sense for $j=\infty$.
\end{rem}

Several technical statements about $j$-independent actions were proved in~\cite{AyzMasEquiv}.

\begin{lem}[{\cite[Lm.3.1]{AyzMasEquiv}}]\label{lemTechJgeneral}
Consider a $j$-independent action of $T$ on $X$, $j\geqslant 1$. Let $X_F$ be a face submanifold. Then
\begin{enumerate}
 \item $\com X_F\leqslant \com X$;
 \item If $\rk F<j$, then $\com X_F=0$;
 \item If $\rk F\geqslant j$, then the action of $T/T_F$ on $X_F$ is $j$-independent.
\end{enumerate}
\end{lem}

\begin{con}
Let $F$ be a face of $Q$ for an action of $T$ on $X$. For simplicity, we denote by $F_{-1}$ the union of lower-rank subfaces of $F$. In the case of equivariantly formal actions the set $F_{-1}$ can be equivalently defined by
\[
F_{-1}=\{x\in F\mid \dim T_x>\rk F\}.
\]
Indeed, every $r$-dimensional orbit is contained in some invariant submanifold (hence a face submanifold in the equivariantly formal case) of rank $r$.

Similarly, we define $(X_F)_{-1}$ as the union of lower-rank face submanifolds. In equivariantly formal case this subset coincides with $\{x\in X_F\mid \dim T_x>\rk F\}$.
\end{con}

In the following propositions we assume that the a coefficient ring $R$ is $\Zo$ if all stabilizers are connected, or $\Qo$ otherwise.

\begin{prop}[{\cite{MasPan}}]\label{propMasPan}
Consider an equivariantly formal $T$-action on $X$ of complexity $0$. Then, for each face $F$ (including $Q$ itself), the following statements hold true.
\begin{enumerate}
 \item $\dim F=\rk F$;
 \item $\Hr^*(F;R)=0$;
 \item $H^i(F,F_{-1};R)=0$ for $i\neq \rk F$ and $H^{\rk F}(F,F_{-1};R)\cong R$.
\end{enumerate}
In other words, each face $F$ is a homology cell, and the filtration $\{Q_i\}$ is a homology cell complex.
\end{prop}

Combining Lemma~\ref{lemTechJgeneral}, Proposition~\ref{propMasPan} and some inductive arguments we get the following statement.

\begin{prop}\label{propAcyclicityMain}
Assume that the action of $T=T^k$ on $X=X^{2n}$ is equivariantly formal and $j$-independent, $j\geqslant 1$.
The faces $F$ of the orbit space $Q=X/T$ have the following homological properties.
\begin{enumerate}
 \item $H^i(F,F_{-1})=0$ for $i<\rk F$.
 \item If $\rk F<j$, then $H^*(F,F_{-1})\cong H^*(D^{\rk F}, \dd D^{\rk F})$. If $\rk F=j$, then $H^i(F,F_{-1})$ vanishes for $i<j$ and $i=j+1$.
 \item $Q$ is $(j+1)$-acyclic (that is $\Hr^i(Q)=0$ for $i\leqslant j+1$).
 \item $Q_r$ is $\min(r-1,j+1)$-acyclic.
\end{enumerate}
\end{prop}

\begin{proof}
Item (1) is proved in~\cite[Lm.3.3]{AyzMasEquiv}. Item (2) is proved in~\cite[Lm.3.2]{AyzMasEquiv}. Item (3) is~\cite[Thm.2]{AyzMasEquiv}. Item (4) is not stated in~\cite{AyzMasEquiv} explicitly, however, its proof follows the same lines as the proof of item (3). We outline the main ideas.

Consider the cohomology spectral sequence associated with the filtration
\begin{equation}\label{eqFiltrQbig}
Q_0\subset Q_1\subset\cdots\subset Q_{r-1}\subset Q_r\subset\cdots\subset Q_k=Q,
\end{equation}
that is $E_1^{p,q}=H^{p+q}(Q_p,Q_{p-1})$. Notice that $H^*(Q_p,Q_{p-1})\cong \bigoplus_{F\mid \rk F=p}H^*(F,F_{-1})$. Items 1 and 2 imply the vanishing of $E_1$ as shown on Fig.~\ref{figSpecSec}.

\begin{figure}[h]
\begin{center}
\begin{tikzpicture}[scale=0.7]
        \draw[->]  (0,0)--(9,0);
        \draw[->]  (0,-1.5)--(0,5);
        \draw (1,0)--(1,1); \draw (2,0)--(2,1); \draw (4,0)--(4,4); \draw (5,0)--(5,4);
        \draw (6,0)--(6,3); \draw (7,0)--(7,2); \draw (8,0)--(8,1);

        \draw (0,1)--(8,1); \draw (4,2)--(7,2); \draw (4,3)--(6,3); \draw (4,4)--(5,4);

        \draw (0.5,0.5) node{$\ast$}; \draw (1.5,0.5) node{$\ast$}; \draw (3,0.5) node{$\cdots$};

        \draw (4.5,4.5) node{$\vdots$};
        \draw (4.5,3.5) node{$\ast$}; \draw (5.5,3.5) node{$\iddots$};
        \draw (4.5,2.5) node{$\ast$}; \draw (5.5,2.5) node{$\ast$}; \draw (6.5,2.5) node{$\iddots$};
        \draw (4.5,1.5) node{$0$}; \draw (5.5,1.5) node{$\ast$}; \draw (6.5,1.5) node{$\ast$}; \draw (7.5,1.5) node{$\iddots$};
        \draw (4.5,0.5) node{$\ast$}; \draw (5.5,0.5) node{$\ast$}; \draw (6.5,0.5) node{$\ast$}; \draw (7.5,0.5) node{$\ast$}; \draw (8.5,0.5) node{$\cdots$};


        \draw (1,4) node{$E_1^{p,q}$};
        \draw (0.5,-0.3) node{\tiny $0$}; \draw (4.5,-0.3) node{\tiny $j$}; \draw (-0.3,0.5) node{\tiny $0$};
        \draw (8.5,-0.3) node{$p$}; \draw (-0.3,4.5) node{$q$};

        \draw (2,2) node{\large $0$}; \draw (2,-1) node{\large $0$}; \draw (6,-1) node{\large $0$};
\end{tikzpicture}
\end{center}
\caption{The first page of the spectral sequence.} \label{figSpecSec}
\end{figure}
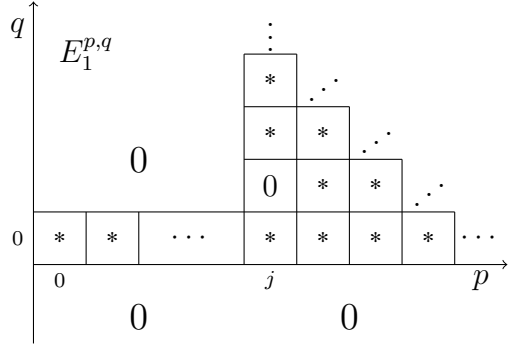

The 0-th row $(E^{p,0}_1,d^1)$ coincides with the differential complex
\begin{multline}\label{eqABdeg0nonAug}
0\to H^0_T(X_0)\stackrel{\delta_0}{\to}
H^{1}_T(X_1,X_0)\stackrel{\delta_1}{\to}\cdots\\\cdots
\stackrel{\delta_{k-2}}{\to}H^{k-1}_T(X_{k-1},X_{k-2})\stackrel{\delta_{k-1}}{\to}H^{k}_T(X,X_{k-1})\to 0.
\end{multline}
which is the degree 0 part of the non-augmented version of the sequence of Atiyah, Bredon, Franz, and Puppe
\begin{multline}\label{eqABseqForX}
0\to H^*_T(X)\stackrel{i^*}{\to} H^*_T(X_0)\stackrel{\delta_0}{\to}
H^{*+1}_T(X_1,X_0)\stackrel{\delta_1}{\to}\cdots\\\cdots
\stackrel{\delta_{k-2}}{\to}H^{*+k-1}_T(X_{k-1},X_{k-2})\stackrel{\delta_{k-1}}{\to}H^{*+k}_T(X,X_{k-1})\to 0.
\end{multline}
(see details in~\cite{AyzMasEquiv}). The sequence~\eqref{eqABseqForX} is acyclic for equivariantly formal actions according to \cite{Bredon} (for rational coefficients) and Franz--Puppe \cite{FP} (over integers presuming connectedness of stabilizers). Therefore, when passing from $E_1^{*,*}$ to $E_2^{*,*}=H(E_1^{*,*},d^1)$ the whole 0-th row disappears except for $E_2^{0,0}\cong H^0_T(X)\cong R$. So far, the second page has the form shown on Fig.~\ref{figSpecSec2}

\begin{figure}[h]
\begin{center}
\begin{tikzpicture}[scale=0.7]
        \draw[->]  (0,0)--(9,0);
        \draw[->]  (0,-1.5)--(0,5);
        \draw (1,0)--(1,1); \draw (2,0)--(2,1); \draw (4,0)--(4,4); \draw (5,0)--(5,4);
        \draw (6,0)--(6,3); \draw (7,0)--(7,2); \draw (8,0)--(8,1);

        \draw (0,1)--(8,1); \draw (4,2)--(7,2); \draw (4,3)--(6,3); \draw (4,4)--(5,4);

        \draw (0.5,0.5) node{$R$}; \draw (1.5,0.5) node{$0$}; \draw (3,0.5) node{$\cdots$};

        \draw (4.5,4.5) node{$\vdots$};
        \draw (4.5,3.5) node{$\ast$}; \draw (5.5,3.5) node{$\iddots$};
        \draw (4.5,2.5) node{$\ast$}; \draw (5.5,2.5) node{$\ast$}; \draw (6.5,2.5) node{$\iddots$};
        \draw (4.5,1.5) node{$0$}; \draw (5.5,1.5) node{$\ast$}; \draw (6.5,1.5) node{$\ast$}; \draw (7.5,1.5) node{$\iddots$};
        \draw (4.5,0.5) node{$0$}; \draw (5.5,0.5) node{$0$}; \draw (6.5,0.5) node{$0$}; \draw (7.5,0.5) node{$0$}; \draw (8.5,0.5) node{$\cdots$};


        \draw (1,4) node{$E_2^{p,q}$};
        \draw (0.5,-0.3) node{\tiny $0$}; \draw (4.5,-0.3) node{\tiny $j$}; \draw (-0.3,0.5) node{\tiny $0$};
        \draw (8.5,-0.3) node{$p$}; \draw (-0.3,4.5) node{$q$};

        \draw (2,2) node{\large $0$}; \draw (2,-1) node{\large $0$}; \draw (6,-1) node{\large $0$};
\end{tikzpicture}
\end{center}
\caption{The second page of the spectral sequence.} \label{figSpecSec2}
\end{figure}

This implies $E_\infty^{p,q}=0$ for $0<p+q\leqslant j+1$ and therefore $\Hr_i(Q)=0$ for $i\leqslant j+1$ which proves item (3) of the proposition.

To prove item (4), consider the spectral sequence associated with the filtration~\eqref{eqFiltrQbig} cut at $r$-th term. In this case, passing from $E_1$ to $E_2$ may result in additional nonzero entry $E^{r,0}_2$ at the rightmost position of $0$-th row. If $r\geqslant j+2$, we still have $E_\infty^{p,q}=0$ for $p+q\leqslant j+1$ and therefore $\Hr_i(Q_r)=0$ for $i\leqslant j+1$. Otherwise, if $r<j+2$, the vanishing in the spectral sequence only implies $(r-1)$-acyclicity of~$Q_r$. This completes the proof of item~(4).
\end{proof}

\begin{cor}\label{corFacesAcyclic}
If a $T$-action on $X$ is equivariantly formal and $j$-independent, then each face $F\subset Q=X/T$ is $(j+1)$-acyclic.
\end{cor}

\begin{proof}
If $\rk F<j$ then $\com X_F=0$ (by Lemma~\ref{lemTechJgeneral}), so $F=X_F/T$ is acyclic by Proposition~\ref{propMasPan}. If $\rk F\geqslant j$, then the induced action on $X_F$ is $j$-independent (again by Lemma~\ref{lemTechJgeneral}), so $F=X_F/T$ is $(j+1)$-acyclic by Proposition~\ref{propAcyclicityMain}.
\end{proof}

\section{Topology of face posets}\label{secPosets}

Recall that $S(X)$ denotes the face poset of $Q$ (or the poset of face submanifolds in $X$) ordered by inclusion. Let $S(X)_r$ denote the subposet of all faces of rank $\leqslant r$.

The symbol $|S|$ denotes the geometrical realization of a poset $S$ that is the geometrical realization of the order complex $\ord S$ (the simplicial complex whose simplices are chains in $S$). An almost immediate corollary from Proposition~\ref{propAcyclicityMain} is the following.

\begin{cor}\label{corAcyclOfSQ}
Assume that $T$-action on $X$ is equivariantly formal and $j$-in\-de\-pen\-dent. Then the geometrical realization $|S(X)_{r}|$ is $(r-1)$-acyclic for $r<j$.
\end{cor}

\begin{proof}
As was noticed earlier, if $r<j$, the orbit type filtration on $Q_r$ is a homological cell filtration with regular cells (essentially due to Proposition~\ref{propMasPan}). Standard arguments with the spectral sequences are used to prove the isomorphisms $H_*(|S(X)_{r}|)\cong H_*(Q_r)$ for regular homological cell complexes (see e.g.~\cite[Prop.5.14]{MasPan} or \cite[Prop.2.7]{Ay1}). The rest follows from item (4) of Proposition~\ref{propAcyclicityMain}.
\end{proof}

With a bit more complicated arguments we can prove a stronger statement.

\begin{thm}\label{thmAcyclAll}
Assume that $T$-action on $X$ is equivariantly formal and $j$-in\-de\-pen\-dent. Then the geometrical realization $|S(X)_{r}|$ is $\min(r-1,j+1)$-acyclic for any~$r$.
\end{thm}

To prove this result, we recall several useful statements about homotopy colimits.

\begin{con}\label{conDiagrams}
Let $S$ be a finite poset, and $\cat(S)$ the finite category, whose objects are elements $s\in S$ and there is exactly one morphism $s_1\to s_2$ if $s_1\leqslant s_2$ (and no morphisms otherwise). An $S$-shaped topological diagram is a functor $D\colon\cat(S)\to \Top$ to the category of topological spaces. Two topological spaces can be associated with each topological diagram $D$: the colimit $\colim_S D$ and homotopy colimit $\hocolim_SD$. Colimit is a synonym for the direct limit of a diagram in the category of topological spaces. Homotopy colimit is the modified version of the colimit, well behaved under homotopy equivalences. The accessible exposition of homotopy colimits and their use in combinatorial topology can be found in~\cite{WZZ}.

There exists a constant diagram $\ast\colon\cat(S)\to\Top$, which maps each $s\in S$ to a point $\pt$. From the definitions it easily follows that $\colim_S\ast$ is a finite set of points corresponding to connected components of $|S|$, while $\hocolim_S\ast=|S|$.
\end{con}

\begin{con}\label{conDiagramOfFaces}
Let $T$ act on a smooth manifold $X$. Consider the diagram $D_Q\colon \cat(S(X))\to\Top$, which maps each face $F$ (as an abstract element of the face poset $S(X)$) to the face $F$ (as a topological space) with morphisms --- the natural inclusions of faces. Since the poset $S(X)$ has the greatest element (the space $Q$ itself), we have $\colim_{S(X)}D_Q=Q$.
\end{con}

Smooth toric actions always admit equivariant cell structures~\cite{Illman2}. This implies that inclusions of subfaces $F_1\hookrightarrow F_2$ of the orbit space $Q=X/T$ admit cellular structures, hence they are cofibrations. Moreover, this argument shows that the diagram $D_Q$ is cofibrant. Therefore,
\begin{equation}\label{eqHocolimIsColim}
\hocolim_{S(X)} D_Q\simeq \colim_{S(X)}D_Q.
\end{equation}

Let $t$ be a fixed nonnegative integer.

\begin{defin}\label{definTequival}
A map $\psi\colon X\to Y$ of topological spaces is called a $t$-equivalence, if the induced map $\psi_*\colon \pi_r(X,b)\to\pi_r(Y,\psi(b))$ is an isomorphism for all $r<t$, and surjective for $r=t$, and for all basepoints $b$.
\end{defin}

\begin{lem}[{Strong Homotopy Lemma~\cite[Lm.2.8]{BjWW}}]\label{lemStrongHomotopyLm}
Let $D_1$, $D_2$ be $S$-shaped diagrams. Let $\alpha\colon D_1\to D_2$ be a map of diagrams such that for each $s\in S$, the map $\alpha_s\colon D_1(s)\to D_2(s)$ is a $t$-equivalence. Then the induced map from $\hocolim_SD_1$ to $\hocolim_SD_2$ is a $t$-equivalence.
\end{lem}

As usual, this statement has a homological version.

\begin{defin}\label{definHomolTequiv}
A map $\psi\colon X\to Y$ is called a homological $t$-equivalence (over coefficient ring $R$), if the induced map $\psi_*\colon H_r(X;R)\to H_r(Y;R)$ is an isomorphism for all $r<t$ and surjective for $r=t$.
\end{defin}

\begin{lem}\label{lemStrongHomotopyHomologyLm}
Let $D_1$, $D_2$ be $S$-shaped diagrams. Let $\alpha\colon D_1\to D_2$ be a map of diagrams such that for each $s\in S$, the map $\alpha_s\colon D_1(s)\to D_2(s)$ is a homological $t$-equivalence. Then the induced map from $\hocolim_SD_1$ to $\hocolim_SD_2$ is a homological $t$-equivalence.
\end{lem}

Although the arguments of~\cite{BjWW} used to prove Lemma~\ref{lemStrongHomotopyLm} work for homology version, we provide an alternative proof based on spectral sequences. Recall that any diagram of spaces (CW-complexes) induces the spectral sequence.

\begin{prop}[{\cite[Prop.15.12]{Dugger}}]\label{propSpecSeqHocolim}
Let $D\colon I\to \Top$ be a diagram over a small category $I$, and $h_*(\cdot)$ --- a generalized homology theory. Then there is a spectral sequence
\[
E_{p,q}^2=H_p(I;h_q(D)) \Rightarrow h_{p+q}(\hocolim_ID).
\]
The differentials have the form $d_r\colon E_{p,q}^r\to E^r_{p-r,q+r-1}$.
\end{prop}

Here $h_q(D)$ denotes the diagram of abelian groups obtained by applying the functor $h_q(\cdot)$ to the topological diagram $D$ element-wise. The module $H_p(I;\ca{A})$ denotes the homology of a small category $I$ with coefficients in a functor $\ca{A}$, which can be defined by one of the equivalent constructions listed below.
\begin{enumerate}
 \item $H_p(I;\cdot)=\underrightarrow{\lim}_p(\cdot)$ is the $p$-th left derived functor of the direct limit functor
 \[
     \underrightarrow{\lim}\colon \Funct(I,\Ab)\to\Ab,
 \]
 where $\Ab$ is the category of abelian groups and $\Funct(I,\Ab)$ is the category of $I$-diagrams of abelian groups.
 \item $H_p(I;\ca{A})$ is the homology of the chain complex
 \[
 C_p(I;\ca{A})=\bigoplus\limits_{x_0\to\cdots\to x_p}\ca{A}(x_0),
 \]
 defined on the nerve of the category $I$.
\end{enumerate}
For equivalence of these constructions we refer to~\cite{QuilAlgK}. Let us prove Proposition~\ref{propSpecSeqHocolim}.

\begin{proof}
We can now prove Lemma~\ref{lemStrongHomotopyHomologyLm} by applying Proposition~\ref{propSpecSeqHocolim}. The diagram map $\alpha\colon D_1\to D_2$ induces the morphism of spectral sequences
\[
\xymatrix{
E_{p,q}^2(D_1) \ar@{=}[r] \ar@{->}[d]^{\alpha_*} & H_p(I;H_q(D_1))  \ar@{=>}[r] \ar@{->}[d] & H_{p+q}(\hocolim_ID_1) \ar@{->}[d]^{\alpha_*} \\
E_{p,q}^2(D_2) \ar@{=}[r] & H_p(I;H_q(D_2))  \ar@{=>}[r] & H_{p+q}(\hocolim_ID_2)
}
\]
Since $\alpha\colon D_1(s)\to D_2(s)$ is a homology $t$-equivalence for each entry $s\in S$, the induced map $\alpha_*\colon E_{p,q}^r(D_1)\to E_{p,q}^r(D_2)$ is an isomorphism for $p+q<t$ or $(p+q=t) \& (q<t)$, while it is surjective for $(p,q)=(0,t)$. This can be proved inductively in $r$, the index of the page. Finally, this implies that $\alpha_*\colon H_{p+q}(\hocolim_ID_1)\to H_{p+q}(\hocolim_ID_2)$ is an isomorphism for $p+q<t$ and surjective for $p+q=t$. Hence $\alpha$ induces a $t$-equivalence of homotopy colimits.
\end{proof}

Now we prove Theorem~\ref{thmAcyclAll}.

\begin{proof}
Let $D_Q\colon\cat(S(X))_r\to\Top$ be the diagram of faces, described in Construction~\ref{conDiagramOfFaces}, and $\ast\colon \cat(S(X))_r\to\Top$ be the constant diagram (which maps every element to a single point). We have a natural morphism of diagrams $\alpha\colon D_Q\to\ast$. Each face $F\in S(X)$ is $(j+1)$-acyclic by Corollary~\ref{corFacesAcyclic}. Therefore $\alpha$ is a $(j+2)$-equivalence on each entry of the diagram. Then Lemma~\ref{lemStrongHomotopyHomologyLm} states that the induced map
\[
\hocolim_{S(X)_r}D_Q\to \hocolim_{S(X)_r}\ast
\]
is a homology $(j+2)$-equivalence. However $\hocolim_{S(X)_r}D_Q\simeq \colim_{S(X)_r}D_Q$ since $D_Q$ is cofibrant. The colimit $\colim_{S(X)_r}D_Q$ is homeomorphic to the $r$-skeleton $Q_r$ by construction. The space $Q_r$ is $\min(r-1,j+1)$-acyclic by Proposition~\ref{propAcyclicityMain}. Therefore $\hocolim_{S(X)_r}\ast \cong |S(X)_r|$ is $\min(r-1,j+1)$-acyclic as well.
\end{proof}

These arguments prove item 1 of Theorem~\ref{thmMainIntroAcycl} from the introduction. Item 2 follows easily, since $S(X)_{\leqslant s}$ is naturally isomorphic to $S(Y)$, whenever $Y$ is a face submanifold corresponding to $s\in S(X)$. This finishes the proof Theorem~\ref{thmMainIntroAcycl}.

\section{Actions of complexity one in general position} \label{secGKMcomb}

In this section we prove Theorem~\ref{thmMainIntroFaceRing} from the Introduction. We apply a very particular case of acyclicity argument to describe the equivariant cohomology algebra of a manifold $X^{2n}$ with an equivariantly formal $(n-1)$-independent action of $T^{n-1}$, when $n\geqslant 5$. In this particular case the description boils down to the theory of Gorenstein face algebras, similar to the complexity zero actions studied in~\cite{MasPan}. As in the case of complexity zero, we start with the GKM description of equivariant cohomology.

Let us recall the basics of GKM theory (see details in~\cite{GKM,Kur}). While usually GKM manifolds refer to complex algebraic varieties with the action of an algebraic torus, we deal with the topological version of the GKM theory.

\begin{defin}\label{defGKMmfd}
A $2n$-dimensional (orientable connected) compact manifold $X$ with an action of $T=T^k$ is called \emph{a GKM manifold} (named after Goresky--Kott\-witz--Mac\-Pherson), if the following conditions hold:
\begin{enumerate}
\item $X$ is equivariantly formal;
\item The fixed point set $X_0=X^T$ is finite and nonempty;
\item The action is 2-independent.
\end{enumerate}
\end{defin}

The next proposition is often taken as a definition of a GKM manifold and is quite standard.

\begin{prop}\label{propGKMspheres}
The 1-dimensional equivariant skeleton $X_1$ of a GKM manifold is a union of $T$-invariant 2-spheres. Each invariant 2-sphere connects 2 fixed points.
\end{prop}

\begin{cor}
The 1-skeleton $Q_1=X_1/T$ is a graph on the vertex set $Q_0\cong X_0$, whose edges correspond to 2-spheres between fixed points.
\end{cor}

Let $\str(p)$ denote the set of edges emanating from a given vertex $p$ of a graph. If $e\in\str(p)$ is an edge emanating from a fixed point $p\in Q_0$ to a fixed point $q$, this edge comes equipped with the weight $\alpha(pq)\in\Hom(T^k,T^1)$. This weight corresponds to the summand of the tangent representation $\tau_pX$ which is tangent to the invariant 2-sphere corresponding to $e$. It easily follows that $\alpha(pq)=\pm\alpha(qp)$.

\begin{defin}
A GKM graph $\G$ is a finite $n$-valent regular graph $(V,E)$ equip\-ped with a function $\alpha\colon E\to \Hom(T^k,T^1)$, which satisfies $\alpha(pq)=\pm\alpha(qp)$ for all edges $e=(pq)$. The function $\alpha$ is called \emph{an axial function}. The numbers $k$ and $n$ in the definition are called \emph{the rank} and \emph{the dimension} of a GKM graph $\G$.
\end{defin}

\begin{defin}
A \emph{GKM graph $\G$ with connection} is a GKM graph, equipped with additional data, the connection. \emph{A connection} $\theta$ is a collection of bijections $\theta_{(pq)}\colon \str(p)\to \str(q)$ for all edges $e=(pq)$ of a graph, satisfying the properties:
\begin{enumerate}
 \item $\theta_{e}e=e$ for all edges $e$;
 \item $\theta_{(qp)}^{-1}=\theta_{(pq)}$;
 \item The integral vector $\alpha(\theta_{(pq)}e)-\alpha(e)$ is collinear to $\alpha(pq)$ for any $e\in\str(p)$.
\end{enumerate}
\end{defin}

\begin{prop}[{\cite[Thm.3.4]{BGH}}]\label{propCanonConnection}
If $X$ is a GKM manifold, its 1-skeleton $Q_1$ is a GKM graph. If, moreover, the action on $X$ is $3$-independent, then $Q_1$ is a GKM graph equipped with a canonical connection.
\end{prop}

Let $\G(X)$ denote the GKM graph corresponding to a torus action on a manifold $X$. The definitions of $j$-independency and the faces of the action inspire the following analogues for abstract GKM graphs.

\begin{defin}
A GKM graph $\G$ is called \emph{$j$-independent} if, for any vertex $p$ of $\G$ the axial values of any $\leqslant j$ edges of $\str(p)$ are linearly independent over $\Qo$.
\end{defin}

Proposition~\ref{propCanonConnection} implies that if $j\geqslant 3$, then the connection $\theta$ on a $j$-independent GKM graph is uniquely determined.

\begin{defin}[{\cite[Def.1.4.2]{GZ}}]
Let $\G$ be an abstract GKM graph with connection. A connected subgraph $\G'\subset \G$ is called \emph{a totally geodesic face} of rank $r$ if it is a GKM graph of rank $r$ and for any edge $pq\in \G'$ there holds $\theta_{pq}(\str(p)\cap \G')=\str(q)\cap \G'$.
\end{defin}

In the following we use the term \emph{face of a GKM graph} instead totally geodesic face for brevity. If a face has dimension $d$, we call it a $d$-face. If a face of a graph has codimension one, it is called a facet.

\begin{con}
If $Y$ is a face submanifold of a GKM manifold $X$, then the graph $\G(Y)$ is naturally a face of the graph $\G(X)$. We call such face \emph{a geometric face} of the GKM graph $\G(X)$. Not every face of $\G(X)$ is necessarily geometric. The simplest example is the full flag manifold $\Fl_3$: its GKM graph has 3 non-geometric totally geodesic faces, see~\cite[Fig.2]{AyzCherep}.

Let $S(\G)$ denote the poset of all faces of $\G$ ordered by inclusion. Although this poset does not coincide with the poset $S(X)$ of geometric faces in general, the poset $S(X)$ can be reconstructed from the GKM graph $\G(X)$ as described in~\cite{AyzCherep}.
\end{con}

\begin{lem}\label{lemm:spans}
Consider a $j$-independent GKM action on a manifold $X$, $j\geqslant 3$, so that $\GX$ is a $j$-independent GKM graph with connection. Then any $\leqslant j-1$ edges emanating from a common vertex span a unique face of $\GX$. This face is geometric.
\end{lem}

The proof can be found e.g. in~\cite[Prop.5.7]{AyzCherep}.

\begin{defin}
An edge $e$ of $\G$ emanating from a vertex of a face $H$ is called \textit{transversal} to $H$ if $e$ is not an edge of $H$.
\end{defin}

Consider $n$-valent $j$-independent abstract GKM graph $\G$. If $j\geqslant 3$, then any two edges emanating from a common vertex determine a 2-face $\Xi$. Notice that combinatorially a $2$-face is a cycle graph, so we have \emph{the monodromy map}: the composition of the connection maps along the edges of the cycle. This monodromy acts on the transverse $n-2$ edges to any given vertex of $\Xi$.

\begin{lem} \label{lemm:1-1}
If a GKM graph is $j$-independent, and $j\geqslant 4$, then the monodromy map along any 2-face acts identically on transverse edges.
\end{lem}

\begin{proof}
Take an arbitrary transversal edge $e$ to $\Xi$ at a vertex $p$. This transversal edge together with the two edges of $\Xi$ emanating from $p$ determines a 3-face since $j\geqslant 4$. Therefore, if we translate $e$ along $\Xi$, it comes back to $e$ (since it stays inside a 3-face). Since $e$ is arbitrary, this proves the lemma.
\end{proof}

The condition $j\geqslant 4$ in Lemma~\ref{lemm:1-1} cannot be weakened. The monodromy map $\mu_F$ is not necessarily trivial when $j=3$ as evidenced by the following examples.

\begin{ex}\label{exGrassmann}
Consider the natural torus action of $T^3$ on the complex Grassmann manifold $\Gr_{4,2}$ of $2$-planes in $\Co^4$. This is a complexity one action in general position. Its GKM graph is shown on Fig.~\ref{figMonod}, (a). This graph is embedded in $\Ro^3$ as a skeleton of an octahedron, and the values of the axial function correspond to the actual geometrical directions of edges in $\Ro^3$. A triangular face of an octahedron  corresponds to a face submanifold $\CP^2$ inside $\Gr_{4,2}$, it is a 2-face in the GKM sense. There are precisely two transversal edges to a face in each vertex. It can be seen that the monodromy along a triangular face transposes the transversal edges.
\end{ex}

\begin{ex}\label{exHP}
There is a canonical action of $T^3$ on the quaternionic projective plane $\HP^2$. This is a complexity one action in general position similar to the previous example. The detailed analysis of the faces of this action was done in~\cite{AyzHP}. The GKM graph is shown schematically on Fig.~\ref{figMonod}, (b), it has 3 vertices, each two connected with a pair of edges. Again, the monodromy along any triangular face permutes its transversal edges. It is also true that the monodromy along any biangle permutes the transversal edges.
\end{ex}

\begin{figure}[h]
\begin{center}
\includegraphics[scale=0.4]{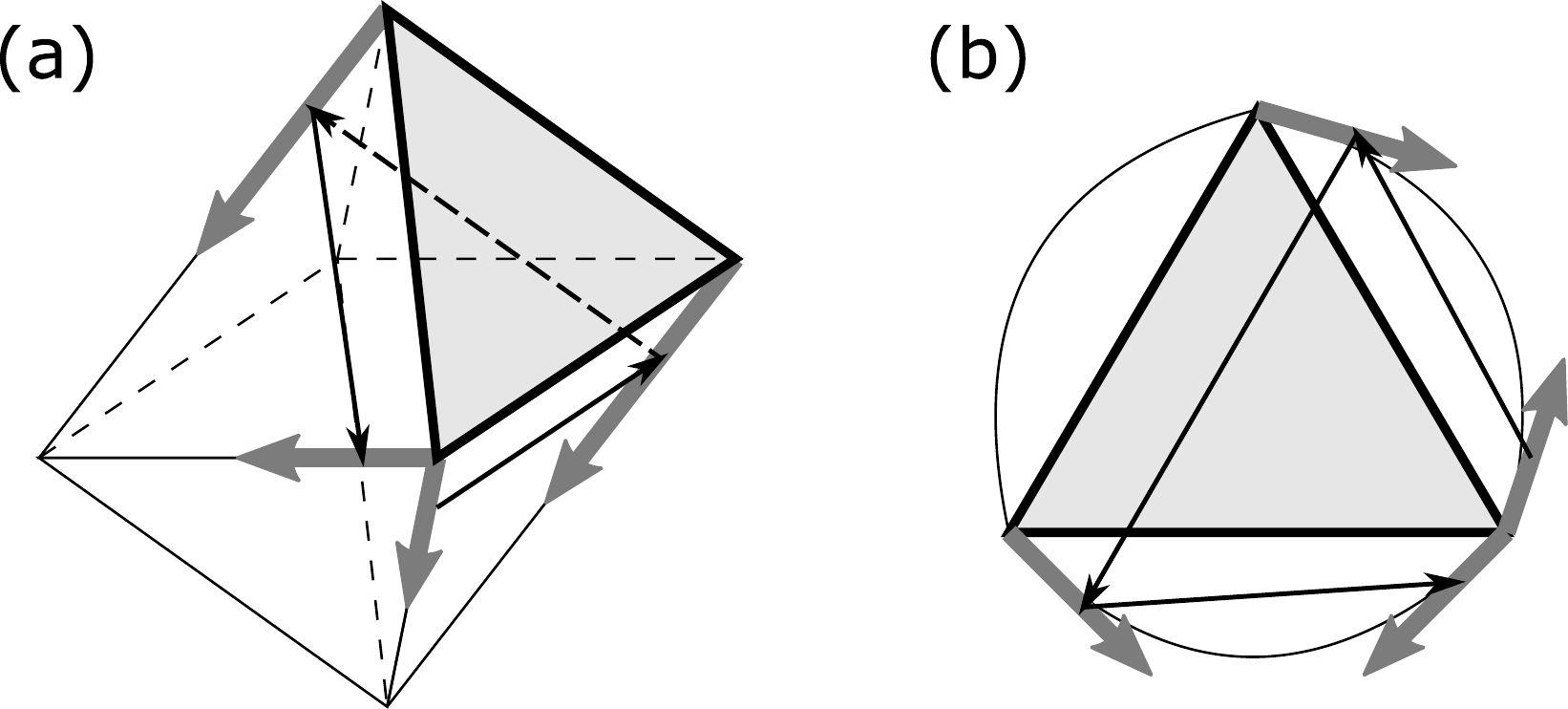}
\end{center}
\caption{The GKM graphs (a) $\G(\Gr_{4,2})$, the Grassmann manifold, (b) $\G(\HP^2)$, quaternionic projective plane. The monodromy along the gray 2-face is nontrivial in both cases.}\label{figMonod}
\end{figure}

Lemma~\ref{lemm:1-1} asserts that the monodromy is trivial on the transversal edges for highly independent actions. The monodromy along a 2-face can, however, be nonidentical on the edges of this face. The next statement is straightforward from the properties of a connection.

\begin{lem}\label{lemm:InnerMonod}
The monodromy along a 2-face is identical on the edges of this face if and only if the face is a cycle of even length.
\end{lem}

Let us introduce several more combinatorial definitions.

\begin{defin}
A GKM graph $\G$ is called \emph{bipartite} if it is bipartite as an unlabelled graph. A graph $\G$ is called \emph{even} every 2-face of $\G$ is a cycle of even length.
\end{defin}

A graph is bipartite if and only if its vertices can be properly colored in 2 colors. Bipartiteness of a graph obviously implies that a graph is even. The converse is also true under the assumption that every closed path in $\G$ is a composition of 2-faces  as an element of $\pi_1(\G)$.

\begin{defin}
An $n$-valent GKM graph $\G$ is called \emph{balanced} if there exists a coloring of its edges in $n$ colors such that
\begin{enumerate}
  \item For any vertex $p$, the edges of $\str(p)$ are colored with all $n$ colors.
  \item Connection preserves the coloring.
\end{enumerate}
\end{defin}

If $\G$ is balanced, then every (totally geodesic) face of $G$ is balanced as well. Therefore the condition of being balanced implies evenness by Lemma~\ref{lemm:InnerMonod}. The aim of the remaining part of the section is to prove the converse statement for complexity one actions in general position.

\begin{prop}\label{propColorability}
Let $X=X^{2n}$ be a simply-connected GKM manifold with smooth effective action of $T=T^{n-1}$ of complexity one in general position, and let $n\geqslant 5$. If the GKM graph $\GX$ is even, then it is balanced.
\end{prop}

%
%

\begin{proof}
The graph $\GX$ is an $(n-1)$-independent GKM graph of dimension $n$ and rank $n-1$. Since $n\geqslant 5$, we are in position to apply Lemma~\ref{lemm:1-1}: an edge $e$ is preserved by the monodromy along any $2$-face which is transversal to $e$. On the other hand, since the graph is even, Lemma~\ref{lemm:InnerMonod} applies as well, so that the monodromy along 2-face preserves $e$ as well if it lies in the face. Therefore the monodromy is trivial along all closed paths from the subgroup of $\pi_1(\GX)$ generated by 2-faces of $\GX$.

Let us prove that 2-faces generate $\pi_1(\GX)$. Consider the orbit space $Q$ of $X$ and the equivariant skeleta $Q_0\subset Q_1\subset Q_2$. Here $Q_1$ is isomorphic to $\GX$ as a graph. $Q_2$ is a homology cell complex according to Proposition~\ref{propAcyclicityMain}. However, in dimensions 1 and 2 every homology cell is an actual cell. The independency assumption on the action implies that $Q_2$ is 1-acyclic. Combining this with simply-connectedness gives $\pi_1(Q_2)=1$. Therefore all closed paths in $Q_1=\GX$ are generated by the boundaries of 2-faces.

This argument shows that monodromy acts trivially. We can color $\str(p)$ in $n$ colors arbitrarily, and then use the monodromy to  transfer it consistently to all other vertices. This procedure determines a proper coloring.
\end{proof}

\begin{defin}
A GKM graph $\G$ is called \emph{a graph with facets}, if, for any vertex $p$ and any edge $e\in\str(p)$ there exists a facet of $\G$ spanned by the edges $\str(p)\setminus\{e\}$.
\end{defin}

\begin{lem}\label{lemColorImpliesFacets}
If $\G$ is balanced, then $\G$ is a graph with facets.
\end{lem}

\begin{proof}
If $\G$ has dimension $n$, with the coloring $c\colon E_\G\to[n]=\{1,\ldots,n\}$, then the facets are the connected components of subgraphs $\G_i=c^{-1}([n]\setminus\{i\})$ for $i\in[n]$.
\end{proof}

Examples~\ref{exGrassmann} and~\ref{exHP} show that GKM graphs $\G(\Gr_{4,2})$ and $\G(\HP^2)$ do not have facets. On the contrast, combining Proposition~\ref{propColorability} and Lemma~\ref{lemColorImpliesFacets} we get the following

\begin{cor}\label{corFacets}
Under the assumptions of Theorem~\ref{thmMainIntroFaceRing} the GKM graph $\GX$ is a graph with facets.
\end{cor}

It follows easily that in the described case all faces of $\GX$ are either geometrical (faces of dimension $\leqslant n-2$ and the whole graph), or the facets provided by Corollary~\ref{corFacets}. In this case the poset $S(\GX)$ is dually simplicial poset, in the sense that every upper order ideal $S(\GX)_{\geqslant s}$ is a boolean lattice. Let $S(\GX)^*$ denote the poset with the reversed order; this is a simplicial poset.

\begin{rem}
We expect that much weaker assumptions are required to guarantee that a GKM graph $\GX$ of an action of $T^{n-1}$ on $X^{2n}$ in general position has facets. In the first version of the paper we stated this fact for $n\geqslant 5$ without the requirement for the graph to be bipartite, but we found a hole in the original proof which we were unable to fix. However, we don't know any counterexamples to this general statement.
\end{rem}

\section{Cohomology and face rings}\label{secFaceRings}

Let us recall the basic theorem used to describe equivariant cohomology ring of a GKM manifold.

\begin{thm}[Theorem of Goresky, Kottwitz, and MacPherson]\label{thmGKM}
Let $X$ be a GKM manifold and $\GX$ its GKM graph with the vertex set $V=X^T$, the edge set $E$ and the axial function $\alpha$.
Consider the $H^*(BT)$-algebra
\[
H_T^*(\GX)\cong \{\phi\colon V\to H^*(BT)\mid
\phi(p)\equiv\phi(q)\mod (\alpha(pq))\,\,\forall pq\in E\},
\]
where the value $\alpha(pq)$ of the axial function is considered as an element of $H^2(BT)$. Then there is a canonical isomorphism of graded $H^*(BT)$-algebras
\[
H_T^*(X)\cong H_T^*(\GX).
\]
\end{thm}

GKM theorem provides an explicit description of $H^*(BT)$ as a subring of the direct sum $\bigoplus_{p\in X^T}H^*(BT)$. An additional work is required if one needs an expression for $H^*(BT)$ in terms of generators and relations. The classical cases are smooth toric varieties and quasitoric manifolds: it is possible to describe their equivariant cohomology rings as Stanley--Reisner algebras. More generally, equivariant cohomology rings of equivariantly formal torus manifolds were described as the face rings of their face posets in~\cite{MasPan}. Here we adopt some ideas of that work for the case of the actions of complexity one in general position under the assumption that their GKM graphs have facets.

In the following, cohomology rings are taken with coefficients in $R=\Qo$. In this section we study a complexity one action of $T=T^{n-1}$ on $X=X^{2n}$ in general position, and assume that $\GX$ has facets. Therefore $S(\GX)^*$ is a simplicial poset as explained in the previous section. For a simplicial poset, there is a well-known notion of the face ring.

\begin{defin}\label{definFaceRing}
Consider \emph{the face ring} $\ZQ[\GX]$ of the simplicial poset $S(\GX)^*$, that is the quotient ring
\[
\ZQ[\GX]=\ZQ[v_F\mid F \text{ a face of }\GX]/\ca{I},
\]
where the ideal $\ca{I}$ is generated by relations
\[
v_Fv_H-v_{F\vee H}\sum\nolimits_{E\subset F\cap H}v_E,\text{ and } v_{\GX}=1,
\]
where $E$ runs over all connected components of the intersection $F\cap H$, and $F\vee H$ denotes the least face of $\GX$ which contains both $F$ and $H$. The ring $\ZQ[\GX]$ is a graded commutative ring with the grading $\deg v_F=2\codim F=2(n-\dim F)$.
\end{defin}

Notice that the element $F\vee H$ is well-defined and unique if $F\cap H\neq \varnothing$. Otherwise, if $F\cap H=\varnothing$, the sum over the empty set of indices is assumed zero, so there is no need to define a unique element $F\vee H$. Elements of $\ZQ[\GX]$ of degree $2$ are called \emph{linear}. The component $\ZQ[\GX]_2$ is generated by $v_F$'s where $F$ is a facet.

\begin{thm}\label{thmFacetsRing}
Let $\GX$ be a as described above. Then there exists a nonzero linear form $\eta\in\ZQ[\GX]_2$ such that $H^*_T(X)$ is isomorphic to $\ZQ[\GX]/(\eta)$.
\end{thm}

The proof in many aspects follows the lines of the similar result about actions of complexity zero given in~\cite{MasPan}. The essential tool of the proof is the notion of the Thom class of a face of GKM graph, which is recalled below. A technical remark is needed for the definition. We impose an omniorientation on a GKM manifold (which means that we orient all its faces), and also an omniorientation on GKM graphs. The latter means that all edges of $\G$ are assumed directed, and the values of the axial function are sensitive to the change of direction, that is $\alpha(pq)=-\alpha(qp)$.

\begin{con}
Let $F$ be a (totally geodesic) face of an omnioriented GKM graph $\G$ with connection on the vertex set $V$. Consider the following element of $\bigoplus_{p\in V}H^*(BT)$ called \emph{the Thom class} of $F$:
\begin{equation}\label{eqThomClass}
\tau_F\colon V\to H^*(BT),\quad \tau_F(p)=\begin{cases}
                                              \prod_{pq\perp F}\alpha(pq), & \mbox{if } p\in F \\
                                              0, & \mbox{otherwise}.
                                            \end{cases}
\end{equation}
The value on a vertex $p$ of $F$ is the product of weights at $p$ transversal to the face~$F$. The element $\tau_F$ is homogeneous of degree $2(n-\dim F)$, twice the codimension of $F$. From the property of the connection on $\G$ it easily follows that $\tau_F$ belongs to the submodule $H^*(\G)\subset \bigoplus_{p\in V}H^*(BT)$.

From this general construction and Theorem~\ref{thmGKM} it follows that whenever $X$ is a GKM manifold, and $F$ is a face of its GKM graph, its Thom class $\tau_F$ is a well-defined element of $H^{2(n-\dim F)}_T(X)$.

Thom class makes sense even if $F$ is not geometric. However, if $F$ is a GKM graph of a face submanifold $Y\subset X$, then $\tau_F$ is the equivariant Poincar\'{e} dual of the submanifold $Y\subset X$. In other words, $\tau_F$ is the image of the identity element under the equivariant Gysin homomorphism $H^0_T(Y)\to H^{2n-2\dim F}(X)$. This fact is easily proved since localizing the element $\tau_F$ to a fixed point $p\in X^T=V$ gives either the Euler class of the normal space to $Y\subset X$ at $p$ (if $p\in Y$), or vanishes (if $p\notin Y$).
\end{con}

Returning back to complexity one actions in general position with facets, one can notice that the elements $\tau_F$ for all faces $F$ of $\GX$ satisfy similar polynomial relations as those in the definition~\ref{definFaceRing} of the face ring.

\begin{lem}
For a complexity one action in general position with facets, the Thom classes of the faces of $\GX$ satisfy the relations
\[
\tau_F \tau_H-\tau_{F\vee H}\sum\nolimits_{E\subset F\cap H}\tau_E,\text{ and } \tau_{\GX}=1.
\]
\end{lem}

The proof easily follows by localizing the relation to each fixed point $p$, see \cite[Lm.6.3]{MasPan}. The assignment $v_F\mapsto \tau_F$ therefore defines a homomorphism
\begin{equation*}
\varphi\colon \ZQ[\GX]\to H^*_T(X).
\end{equation*}

\begin{lem}
The map $\varphi$ is surjective.
\end{lem}

\begin{proof}
The same argument as \cite[Prop.7.4]{MasPan} shows that $H^*_T(X)$ is generated by $\tau_F$'s as a module over $H^*(BT)$ and $H^2(BT)$ is generated by $\tau_G$'s over $\ZQ$ where $G$'s are facets. This implies the lemma.
\end{proof}

\begin{prop} \label{propLinearForm}
Assume that $X$ is a GKM manifold of dimension $2n$, such that the action has complexity one in general position, and $\GX$ has facets. Then there is a non-zero linear form $\eta\in \ZQ[\GX]$ such that $\varphi(\eta)=0$ and $\varphi$ induces an isomorphism
\[
\bv\colon \ZQ[\GX]/(\eta)\to H^*_T(X).
\]
\end{prop}

\begin{proof}
We consider the commutative diagram
\begin{equation} \label{eq:CD}
\begin{gathered}
\xymatrix{
\ZQ[\GX]\ar[r]^\varphi \ar[d]_s & H^*_T(X)\ar[d]^r\\
\bigoplus\limits_{p\in X^T}\ZQ[\GX]/(v_H\mid p\notin H) \ar[r]^(0.61){\psi} & \bigoplus\limits_{p\in X^T} H^*(BT),
}
\end{gathered}
\end{equation}
%
where $\psi$ is induced by $v_F\to \tau_F(p)$ for each $p\in X^T$, and $\psi$ is induced by $v_F\to \tau_F(p)$ for each $p\in X^T$. The vertical map $r$ is injective since the action is equivariantly formal (see e.g. the GKM model given by Theorem~\ref{thmGKM}). The vertical map $s$ is also injective, this follows from the fact that the face algebra is an algebra with straightening law (e.g. see~\cite[Thm.3.5.6]{BPnew} and the remark after that statement).

Notice that each summand $\ZQ[\GX]/(v_H\mid p\notin H)$ in the expression on the left is isomorphic to the polynomial algebra in $n$ generators (see explanation below). Since $\dim T=n-1<n$, the summand $H^*(BT)$ on the right is a polynomial algebra in $n-1$ generators. It follows that the map $\psi$ is surjective but not injective, even in degree two. Therefore the commutativity of the above diagram implies that $\ker\varphi$ has a nonzero linear form $\eta=\sum_{i=1}^m c_i\tau_i$ where $c_i\in \ZQ$ and $\tau_i$'s are the Thom classes of the facets of $\GX$. Then $r(\varphi(\eta))=0$ and hence $\sum_{i=1}^mc_i\tau_{i}(p)=0$ for any $p\in X^T$. Here $\tau_{i}(p)\neq0$ if and only if $p$ is a vertex of the facet corresponding to $\tau_i$.
Since $\tau_{i}(p)$'s span $H^2(BT)$ which is of rank $n-1$, the coefficient vector $(c_1,\dots,c_m)$ is uniquely determined up to scalar multiple. This shows that the epimorphism
\[
\bv\colon \ZQ[\GX]/(\eta)\to H^*_T(X)
\]
induced from $\varphi$ is an isomorphism on degree two.

Note that any $c_i$ is nonzero. Indeed, for any fixed point $p\in X^T$ we have $n$ tangent weights $\alpha_{p,1},\ldots,\alpha_{p,n}\in \Hom(T,T^1)\cong \Zo^{n-1}$ attached to this point. There is a unique (up to multiplier) linear relation on these vectors $\sum_{j\in [n]}c'_j\alpha_{p,j}=0$, and the coefficients $c'_j$ are nonzero, since every $n-1$ of the weights are linearly independent, see details in~\cite{AyzCompl}. Each number $c'_j$ is the number $c_i$ corresponding to the facet transversal to the weight $\alpha_{p,j}$ at $p$.

By moding out the ideals generated by $\eta$ in the commutative diagram \eqref{eq:CD} we get a commutative diagram
\begin{equation} \label{eq:CD2}
\begin{gathered}
\xymatrix{
\ZQ[\GX]/(\eta)\ar[r]^\bv \ar[d]_{\bar{s}} & H^*_T(X)\ar[d]^r\\
\bigoplus\limits_{p\in X^T}\ZQ[\GX]/(\eta)/(v_H\mid p\notin H) \ar[r]^(0.62){\bar{\psi}} & \bigoplus\limits_{p\in X^T} H^*(BT)
}
\end{gathered}
%
\end{equation}
where $\bar{s}$ is injective since so is $s$.

{\it Claim.} $\bar{\psi}$ is an isomorphism. 

Indeed, since $\eta=\sum_{i=1}^mc_i\tau_i$ and $\tau_i(p)\neq0$ if and only if the corresponding facet contains $p$, we have
\begin{equation} \label{eq:1-4}
\ZQ[\GX]/(\eta)/(v_H\mid p\notin H)=\ZQ[\tau_i\mid i\in I(p)]/\left(\sum\nolimits_{i\in I(p)} c_i\tau_i\right)
\end{equation}
where $I(p)=\{i\mid \tau_i(p)\neq0\}$. Since $|I(p)|=n$ and the coefficients $c_i$ are non-zero, the ring in \eqref{eq:1-4} is isomorphic to $H^*(BT)$. This together with the surjectivity of $\bar{\psi}$ implies that $\bar{\psi}$ is an isomorphism, proving the claim.
Since both $\bar{s}$ and $\bar{\psi}$ are injective, the commutativity of the diagram above shows that the epimorphism $\bv$ is indeed injective on any degree, proving the theorem.
\end{proof}

Proposition~\ref{propLinearForm} proves Theorem~\ref{thmFacetsRing}. Combining it with Corollary~\ref{corFacets} from the previous section, we obtain the proof of Theorem~\ref{thmMainIntroFaceRing}.

\begin{rem}\label{remGorenstein}
Since $X$ is equivariantly formal, the equivariant cohomology algebra $H^*_T(X)\cong\ZQ[\GX]/(\eta)$ is a free module over $H^*(BT)$, which is a subalgebra freely generated by some linear forms $\theta_1,\ldots,\theta_{n-1}$. Hence $\theta_1,\ldots,\theta_{n-1}$ is a regular sequence in $\ZQ[\GX]/(\eta)$. Notice that $\eta$ is a regular element in $\ZQ[\GX]$ since its localization to each fixed point $p\in X^T$ is nonzero. Therefore, the face ring $\ZQ[\GX]$ has a regular sequence $\eta,\tilde{\theta}_1,\ldots,\tilde{\theta}_{n-1}$, where $\tilde{\theta}_i$ is a lift of $\theta_i$ in $\ZQ[\GX]$. Since the quotient $\ZQ[\GX]/(\eta,\tilde{\theta}_1,\ldots,\tilde{\theta}_{n-1})\cong H^*(X)$ is finite dimensional, this is a maximal regular sequence. Therefore the face ring $\ZQ[\GX]$ is Cohen--Macaulay. Moreover, since the quotient is a Poincare duality algebra, the ring $\ZQ[\GX]$ is Gorenstein. Therefore, simplicial poset $S(\GX)^*$ is Gorenstein~\cite{Stan}. It is also a Gorenstein* poset since the top-degree component of $\ZQ[\GX]/(\eta,\tilde{\theta}_1,\ldots,\tilde{\theta}_{n-1})\cong H^*(X)$ has degree $2n$. This implies that geometrical realizations of $S(\GX)^*$ and all its links are homology spheres.
\end{rem}

We conclude the paper with an observation which relates cohomology of $X$ with that of its face submanifolds. Proposition~\ref{propLinearForm} implies that if $Y$ is a face submanifold of $X$ such that any intersection of faces of $\GX$ with $\GY$ is connected unless empty, then the restriction map
\begin{equation} \label{eq:2-1}
\iota^*\colon H^*_T(X)\to H^*_T(Y)
\end{equation}
is surjective when $\GX$ has facets. This is not true in general.

\begin{ex}
Consider $X=\Gr_{4,2}$ as a continuation of Example~\ref{exGrassmann}. Let $Y$ a face of $\Gr_{4,2}$ isomorphic to $\CP^1\times \CP^1$. It corresponds to an equatorial square cycle of an octahedron shown on Fig.~\ref{figMonod}, (a). The restriction map $\iota^*\colon H^2_T(\Gr_{4,2})\to H^2_T(\CP^1\times \CP^1)$ is not surjective: the rank of the target module is obviously bigger than the rank of the source. However, in this example surjectivity holds in higher degrees.
\end{ex}

\begin{prop}
Let $X$ be a GKM manifold of complexity one in general position. If any intersection of geometric faces of $\GX$ with $\GY$ is connected unless empty, then $\iota^*$ in \eqref{eq:2-1} is surjective in degrees $\geqslant 4$.
\end{prop}

\begin{proof}
We think of $H^*_T(X)$ and $H^*_T(Y)$ as the cohomology of the GKM graphs $\GX$ and $\GY$. For any face $G$ of $\GY$ which is of codimension $\geqslant 2$, there exists a face $F$ of $\GX$ such that $F$ and $\GY$ intersect transversally in the vertices of~$G$. Since $F$ also has codimension $\geqslant 2$ in $X$, it is a geometric face. By assumption $G=F\cap \GY$, since there are no other connected components in the intersection. Hence $\iota^*(\tau_F)=\tau_G$, proving the proposition.
\end{proof}

\section*{Acknowledgements}
The authors thank Fedor Pavutnitskiy from whom we knew about the works of Quillen on spectral sequences of homotopy colimits. The last author would like to thank Shintar\^o Kuroki for many fruitful discussions of GKM theory. We are also grateful to Ivan Limonchenko who pointed out an inaccuracy in the proof of the second theorem, that appeared in the first version of this text. We thank the anonymous referees for many useful remarks, which helped to improve the exposition.

\end{document}